\documentclass[11pt]{amsart}
\usepackage{amscd}
\usepackage{amsthm}
\usepackage[centertags]{amsmath}
\usepackage{amsfonts}
\usepackage{newlfont}
\usepackage{graphicx}
\usepackage[usenames]{color}
\usepackage{amsfonts, amssymb,accents}
\usepackage{mathrsfs}
\usepackage{latexsym}
\usepackage{bbold}
\usepackage{enumitem}
\usepackage{tikz}\usepackage{verbatim}
\usepackage{tabulary,tabularx}
\usepackage[all]{xy}
\usepackage{tikz}
\usetikzlibrary{positioning}
\usetikzlibrary{shapes.multipart}
\definecolor{colorcita}{RGB}{21,86,130}
\definecolor{colorref}{RGB}{5,10,177}
\definecolor{colorweb}{RGB}{177,6,38}

\newtheorem{theorem}{Theorem}[section]

\newtheorem{proposition}[theorem]{Proposition}
\newtheorem{corollary}[theorem]{Corollary}
\newtheorem{lemma}[theorem]{Lemma}
\newtheorem{remark}[theorem]{Remark}
\newtheorem{definition}[theorem]{Definition}

\newtheorem{question}[theorem]{Question}

\newcommand{\zN}{\mathbb N}

\renewcommand{\H}{H(\mathbb{C})}

\newcommand{\sub}{\subseteq}

\newcommand{\f}{\frac}

\newcommand{\AP}{\mathcal {AP}}
\newcommand{\A}{\mathcal {F}}

    \title{
    Multiple recurrence and hypercyclicity}

\begin{document}

\begin{abstract}
We study multiply recurrent and hypercyclic operators as a special case of $\mathcal F$-hypercyclicity, where $\mathcal F$ is the family  of subsets of the natural numbers containing arbitrarily long arithmetic progressions. We prove several properties  of hypercyclic multiply recurrent operators, we characterize those operators which are weakly mixing and multiply recurrent, and we show that there are operators that are multiply recurrent and hypercyclic but not weakly mixing.
\end{abstract}

\author{Rodrigo Cardeccia, Santiago Muro}
\address{DEPARTAMENTO DE MATEM\'ATICA - PAB I,
	FACULTAD DE CS. EXACTAS Y NATURALES, UNIVERSIDAD DE BUENOS AIRES, (1428) BUENOS AIRES, ARGENTINA AND IMAS-CONICET} \email{rcardeccia@dm.uba.ar} 
\address{FACULTAD DE CIENCIAS EXACTAS, INGENIERIA Y AGRIMENSURA, UNIVERSIDAD NACIONAL DE ROSARIO, ARGENTINA AND CIFASIS-CONICET}
\email{muro@cifasis-conicet.gov.ar}

\thanks{Partially supported by ANPCyT PICT 2015-2224, UBACyT 20020130300052BA, PIP 11220130100329CO and CONICET}

\maketitle
\section { Introduction }


Recurrence is one of the oldest notions in the theory of dynamical systems. It arose at the end of the IXX century with the Poincar\'e Recurrence Theorem. In the seventies Furstenberg introduced the concept of multiple recurrence and proved the Multiple Recurrence Theorems which had a profound impact in dynamical systems, ergodic theory and its applications to number theory and combinatorics.


In the 90's, a systematic study of the dynamics of the linear operators on infinite dimensional spaces began, and it has experienced a lively development in the last decades, see \cite{BayMat09,GroPer11}. The main concept in this theory is that of hypercyclicity: an operator is called hypercyclic if it has a dense orbit.
It has been proved for example that every infinite dimensional and separable Fr\'echet space supports a hypercyclic operator \cite {ansari1997existence,bernal1999hypercyclic}, 
or that there are hypercyclic operators $T$ such that $T\oplus T$ are no longer hypercyclic \cite {dlRRea09} (i.e. $T$ is not weakly mixing).  
Over the last years much of the attention was driven to special types of hypercyclicity like \textit {frequent hypercyclicity} \cite{BayGri06}, \textit{upper frequent hypercyclicity} \cite{Shk09} and more recently to $\A$\textit{-hypercyclicity} \cite{Bes16,BesMen19,BonGro18,bonilla2020frequently}, for more general families $\A$ of subsets of $\mathbb N$.

The notion of recurrent linear operators had not been systematically studied 
until the work Costakis, Manoussos and Parissis in \cite{CosManPar14}. 
Costakis and Parissis \cite{CosPar12} were also the first to study multiple recurrence in the context of linear dynamics. An operator is (topologically) \textit{multiply recurrent} provided that for every open set $U$ and every $m$, there is $k$ such that $\cap_{i=0}^m T^{-ik}(U)\neq \emptyset.$  In \cite{CosPar12}, the authors characterize the bilateral weighted shifts and adjoints of multiplication operators which are 
multiply recurrent and prove a result which assures that for certain sequences of scalars $(\lambda_n)_n$  the frequent hypercyclicity of the sequence  $(\lambda_nT^n)_n$ implies that $T$ itself is multiply recurrent. This notion was also studied in \cite{CheChu162,CheChu16,CheChu18}, where the authors studied different examples of multiply recurrent operators and in \cite{Pui18} where the author studied the relation between multiple recurrence and reiterative hypercyclicity.


In the present note we study multiply recurrent operators from the point of view of $\A$-hypercyclicity. 
Our point of departure is the observation that an operator is multiply recurrent if and only if it is  $\AP$-recurrent 
and that it is multiply recurrent and hypercyclic if and only if it is  $\AP$-hypercyclic, where $\AP$ stands for the family of natural numbers supporting arbitrarily long arithmetic progressions.  The study of sets having arbitrarily long arithmetic progressions 
is a central task in number theory and additive combinatorics. For instance, the celebrated Szemeredi's Theorem \cite{Sze75} and the Green-Tao Theorem \cite{GreTao08} establish that the sets having positive lower density  and  the set of prime numbers belong to $\AP$. 
Our observation is not a surprise, since  the study of sets with large arithmetic progressions has an intrinsic relation to ergodic theory, and it is well known that Szemeredi's Theorem can be proved (and is equivalent to)
Furstenberg's Multiple Recurrence Theorem.


We start showing some basic properties of $\AP$-hypercyclic operators, including a Birkhoff transitivity type result, an Ansari type theorem and the existence of $\AP$-hypercyclic operators on arbitrary separable infinite dimensional Fr\'echet spaces.
In \cite[Section 5]{CosPar12} it was shown that there are multiply recurrent and hypercyclic operators that are not frequently hypercyclic.  Using deep results by Gowers \cite{gowers2001new} and by Bayart and Matheron \cite{BayMat09Non}, we show the existence of multiply recurrent hypercyclic operators which are not even weakly mixing.  We also give some characterizations of the operators that are weakly mixing and $\AP$-hypercyclic, for example, in terms of an $\AP$-hypercyclicity criterion and of $\AP$-hereditary hypercyclicity. 
 In our final section we study $\mathcal F$-hypercyclicity for a related family, which we call $\overline{\AP}$. We show that while 
for a single operator this concept coincides with $\AP$-hypercyclicity, for sequences of operators both concepts differ. This allows us to prove an enhanced version of the main result in \cite{CosPar12}.

Let us recall some basic facts on $\mathcal F$-hypercyclicity. Given a  \textit {hereditary upward} family $\A\sub \mathcal P(\zN)$ (also called Furstenberg family) we say that an operator is $\A$-hypercyclic if there is $x\in X$ for which the sets $N_T(x,U):=\{n\in\zN: T^n(x)\in U\}$ of return times belong to $\A$. Thus, for example, if we take $\A$ to be the family of non empty sets, $\A$-hypercyclicity is simply hypercyclity.
Let us recall the following examples of $\A$-hypercyclicity, which are the most widely studied in the literature:
\begin{itemize}
\item $\underline{\mathcal{D}}=\{$ sets with positive lower density$\}$ (i.e. $A\in  \underline{\mathcal{D}}$ if $\underline {dens} (A):=$ $\liminf_n \# \f{\{k\leq n: k\in A\}}{n}>0$). An operator is frequently hypercyclic if and only if it is $ \underline{\mathcal{D}}$- hypercyclic. 
\item $\overline{\mathcal{D}}=\{$ sets with positive upper density$\}$ (i.e. $A\in  \overline{\mathcal{D}}$ if $\overline {dens} (A):=$ $\limsup_n \# \f{\{k\leq n: k\in A\}}{n}>0$). An operator is upper frequently hypercyclic if and only if it is $ \overline{\mathcal{D}}$- hypercyclic. 
\item $\overline{\mathcal{BD}}=\{$ sets with positive Banach upper density$\}$ (i.e. $A\in  \overline{\mathcal{BD}}$ if  $\lim_n \sup_k\# \f{\# A\cap [k,k+n]}{n}>0$). An operator is reiterative hypercyclic if and only if it is $ \overline{\mathcal{BD}}$- hypercyclic. 
\end{itemize}

For more $\A$-hypercyclicity see \cite{Bes16,BonGro18, bonilla2020frequently, cardeccia2020arithmetic,ErnEssMen21,Shk09,Pui17,Pui18}.
 The following general result was proved in \cite{BonGro18}.
\begin{theorem}[Bonilla-Grosse Erdmann]\label{equivalencias upper}
Let $\A$ be a an upper hereditary upward family and $T$ be a linear operator on a separable Fr\'echet space. Then the following are equivalent:
\begin{enumerate}[label=(\roman*)]
    \item For any open set $V$ there is $\delta$ such that for any open set $U$ there is $x\in U$ with $N_T(x,U)\in \A_\delta.$
    \item For any open set $V$ there is $\delta$ such that for every $U$ and $m$ there is $x\in U$  with $N_T(x,U)\in \A_{\delta,m}$.
    \item The set of $\A$-hypercyclic points is residual.
    \item $T$ is $\A$-hypercyclic.
\end{enumerate}
\end{theorem}
Recall that a hereditary upward family is said to be \textit {upper} provided that $\emptyset\notin \A$ and $\A$ can be written as 
\begin{align}\label{def: upper family}
\bigcup_{\delta\in D} A_\delta,\quad\textrm{with }\; \A_\delta=  \bigcap_{m\in M} \A_{\delta,m},
\end{align}
where $M$ is countable and such that the families $\A_{\delta,m} $ and $\A_{\delta}$ satisfy
\begin{itemize}
\item each $\A_{\delta,m}$ is \textit{finitely hereditary upward}, that means that for each $A\in \A_{\delta,m}$, there is a finite set $F$ such that $F\cap A\sub B,$ then $B\in \A_{\delta,m}$;
\item $\A_\delta$ is \textit{uniformly left invariant}, that is, if $A\in\A$ then there is $\delta$ such that for every $n$, $A-n\in \A_\delta$. 
\end{itemize}
For example the families $\overline{\mathcal D},\overline{\mathcal {BD}}$ are upper while $ \underline{\mathcal{D}}$ is not upper (see \cite{BonGro18}).


\section{$\mathcal{AP}$-hypercyclic operators and multiple recurrence - Basic properties}\label{seccion alap}

An operator is said to be (topologically) \emph{ multiply recurrent} provided that for every open set $U$ and every $m$ there is $k$ such that
$$U\cap T^{-k}(U)\cap\dots \cap T^{-km}(U)\neq \emptyset.$$

Recall that the arithmetic progression of length $m+1$ ($m\in\mathbb N$), common difference $k\in\mathbb N$ and initial term $a\in\mathbb N$ is the subset of $\mathbb N$ of the form $\{a,a+k,a+2k,\dots,a+mk\}$.
 We denote by $\AP$ the (Furstenberg) family of subsets of the natural numbers that contain \textit {arbitrarily long arithmetic progressions}.


 We will see now that multiple recurrence may be studied from the $\A$-hypercyclicity point of view, 
indeed, our next result observes that $\AP$-hypercyclicity is equivalent to multiple recurrence plus hypercyclicity. 
This concept was also recently studied in \cite{kwietniak2017multi} for compact dynamical systems.
The family $\AP$ is an upper Furstenberg family: just let, in \eqref{def: upper family}, $\A_m$ be the family of subsets with arithmetic progressions of length greater than $m$ and let $\A_\delta=\AP$.
Applying Theorem \ref{equivalencias upper} we have the following (see also \cite[Proposition 4.14]{kwietniak2017multi}).
\begin{proposition}\label{equivalencias alap}
Let $T$ be a linear operator on a separable Fr\'echet space. Then the following are equivalent.
\begin{enumerate}[label=(\roman*)]
\item $T$ is hypercyclic and every hypercyclic vector is $\AP$-hypercyclic.
\item  There is an $\AP$-hypercyclic vector.
\item $T$ is hypercyclic and multiply recurrent.
\item  For each pair of nonempty open sets $U,V$ and each $m>0$ there exists $x\in U$ such that $N_T(x,V)$ has an arithmetic progression of length $m$.
\item  The set of $\AP$-hypercyclic vectors is residual. 
\end{enumerate}
\end{proposition}
\begin{proof}
 $(i) \Rightarrow (ii) \Rightarrow (iii) \Rightarrow (iv)$ and $(v)\Rightarrow (ii)$  are all straightforward; also $(iv) \Rightarrow (v)$ is a direct consequence of Theorem \ref{equivalencias upper}.
 
 $(iii)\Rightarrow (i)$. Let $x$ be a hypercyclic vector and $U$ a nonempty open set. Let $m>0$. Thus, there is $k_2$ such that $V:=U\cap T^{-k_2}(U)\cap\dots \cap T^{-k_2m}(U)\neq \emptyset$ and hence $N(x,V)\neq \emptyset.$ If $k_1\in N(x,V)$ it follows that for every $j\leq m$, $T^{k_1+jk_2}(x)\in U$ for every $j\leq m$.
 \end{proof}

Recall that, given a family $\A$ of natural numbers, a vector $x$ is said to be $\A$- recurrent for $T$ provided that for every open set $U$ containing $x$ the set of hitting times  $N_T(x,U)\in \A$ (see \cite{Fur81,bonilla2020frequently}). If the set of $\A$-recurrent vectors is dense, then the operator is said to be $\A$-recurrent. The next result can also be found in \cite[Lemma 4.8]{kwietniak2017multi}.
\begin{proposition}\label{multiple rec implica AP-rec}
Let $T$ be a linear operator.
The following are equivalent:
\begin{itemize}
    \item [(i)] $T$ is multiply recurrent,
    \item [(ii)] $T$ is $\AP$-recurrent.

\end{itemize}
\end{proposition}
\begin{proof}
$(ii) \Rightarrow (i)$ is straightforward.
For the converse, let $U$ be an open set. By the multiple recurrence of $T$ we may construct by induction open balls $U_m$ and steps $k_m$ such that
\begin{itemize}
    \item $\overline{U_m}\sub U_{m-1}\ldots \sub U$;
    \item for every $j\leq m$ $T^{j k_m}( U_m)\sub U_{m}$ and
    \item the radius of $U_m$ tends to zero.
\end{itemize}
Let $x\in \bigcap_m \overline {U_m}$. Then $x$ is an $\AP$-recurrent vector. Indeed, given an open set $V$ containing $x$ and $m>0$ then there is $m'>m$ such that $U_{m'}\sub V$. Hence, for every $j<m'$ we have that $T^{j k_m'}(x)\in U_{m'}\sub V$.
\end{proof}

In \cite{CosPar12}, the authors showed   examples of a hypercyclic bilateral weighted shift on $\ell_p$ (hence weakly mixing)  which is not  multiply recurrent and a bilateral weighted shift  which is   multiply recurrent and hypercyclic but not frequently hypercyclic. Since for weighted backward shifts frequent hypercyclicity is equivalent to reiterative hypercyclicity \cite{Bes16}, it follows that weakly mixing does not imply $\mathcal{AP}$-hypercyclicity and $\mathcal{AP}$-hypercyclicity does not imply reiterative hypercyclicity. We will see in Theorem \ref{alap no weakly} that there are $\AP$-hypercyclic operators that are not weakly mixing.

On the other hand any chaotic operator is 
$\AP$-hypercyclic. 
Moreover, $\AP$-hypercyclicity is also implied by 
reiterative hypercyclicity. This follows from a direct application of Szemeredi's Theorem \cite{Sze75}.
\begin{proposition}
Let $T$ be a reiterative hypercyclic operator. Then $T$ is $\AP$-hypercyclic.
\end{proposition}

A typical problem in $\A$-hypercyclicity is to determine whether $T^{-1}$ and $T^p$ are $\A$-hypercyclic provided that $T$ is $\A$-hypercyclic. For $\AP$-hypercyclicity we obtain an easy answer.
\begin{proposition}
	Let $T$ be an invertible $\mathcal{AP}$-hypercyclic operator on a Fr\'echet space. Then $T^{-1}$ is $\mathcal{AP}$-hypercyclic.
\end{proposition}
\begin{proof}
	Since $T$ is hypercylic it follows that $T^{-1}$ is hypercyclic.
	Let $m\in\zN$ and $U$ be an open set. Since $T$ is $\mathcal{AP}$-hypercyclic there is $x\in U$ and $n\in\mathbb N$ such that $T^{jn}x\in U$ for every $j\leq m$. Let $y= T^{mn}(x)$. It follows that $T^{-jn}(y)=T^{(m-j)n}(x)\in U$ for every $0\leq j \leq m$. 
\end{proof}
Shkarin in \cite[Section 5]{Shk09} proved that for any right shift invariant Ramsey family $\A$, if an operator is $\A$-hypercyclic then so are its powers and its rotations (see also \cite{bonilla2020frequently}). Thus by van Waerden's theorem we have the following.
\begin{proposition}\label{Ansari AP}
	Let $T$ be an $\mathcal{AP}$-hypercyclic operator on a Fr\'echet space. Then $T^p$ and $\lambda T$ are $\mathcal{AP}$-hypercyclic, for any $p\in\mathbb N$ and $|\lambda|=1$. Moreover, they share the $\AP$-hypercyclic vectors.
\end{proposition}

Recall that a set of operators $\{T_1,\dots,T_r\}$ is said to be $d$-transitive (respectively, $d$-mixing) if for every nonempty open set $U$ and every nonempty open sets $U_i$, $1\le1\le r$, there is $n$  such that (respectively, there is $N$ such that for every $n\ge N$) 
$$U\cap \bigcap_{i\le r} T_i^{-n}(U_i)\ne\emptyset.$$
The study of disjointness for tuples of linear operators began in \cite{Ber07,BesPer07}.
It is clear that if an operator $T$ is such that $\{T, T^2, \dots ,T^m\}$ is $d$-transitive for every $m$ then $T$ is $\AP$-hypercyclic.

In \cite{BesMarPerShk12} the authors showed that any operator $T$ such that $T-I$ is a backward shift satisfies that $\{T, T^2, \dots ,T^m\}$ is $d$-mixing for every $m$. Note that this in particular answers \cite[Question 7.1]{CosPar12}.
Since in 
every infinite dimensional and separable Fr\'echet space there exists an operator that is quasiconjugated to the the sum of a weighted shift and the identity on $\ell_1$ (\cite{BonePer98}), we have as a corollary
an existence result for $\AP$-hypercyclic operator.
\begin{corollary}\label{existe ALAP}
Let $X$ be an infinite dimensional separable Fr\'echet space. Then there exists an $\mathcal{AP}$-hypercyclic operator.
\end{corollary}

The most efective tool to prove that an operator is hypercyclic is to show that it satisfies the hypercyclicity criterion. In \cite{Ber07,BesPer07}, the authors introduced $d$-hypercyclic operators. In particular if  $(T, T^2, \dots ,T^m)$ satisfy the $d$-hypercyclicity criterion for every $m$  then  $T$ is $\AP$-hypercyclic. 

Note that an arithmetic progression whose initial term coincides with the common difference is just a set of the form $\{q,2q,\dots,mq\}$ for some $q,m\in\mathbb N$. 
Therefore, by \cite[Section 2]{BesPer07} we have the following.
\begin{proposition}\label{criterio alap}
	Let $X$ be a separable Fr\'echet space and  $T$ an operator such that there exists a dense set $X_0\sub X$, a function $S:X_0\rightarrow X_0$ and a sequence $(m_k)_k\in\AP$,  which contains arbitrarily long arithmetic progressions whose initial term coincides with the common difference such that for each $x\in X_0$,
	\begin{enumerate}
		\item $T^{m_k}(x)\rightarrow 0$;
		\item  $S^{m_k}(x)\to 0$ and
		\item $TS(x)= x$.
	\end{enumerate}
Then $T$ is $\mathcal{AP}$-hypercyclic.
\end{proposition}
Recall that an operator is said to satisfy the strong Kitai's criterion provided that it satisfies the above criterion but for the full sequence of natural numbers. 
\begin{corollary}\label{strong Kitai Criterion}
Every operator that satisfies the strong Kitai's Criterion is $\mathcal{AP}$-hypercyclic.
\end{corollary}

\section {$\AP$-hypercyclic backward shifts}
In \cite{CosPar12} the authors characterized the bilateral weighted backward shifts on $\ell_2$ which are multiply recurrent. They also showed that recurrent bilateral weighted shifts are hypercyclic and hence  every multiply recurrent bilateral weighted shift on $\ell_2$ is in fact $\mathcal{AP}$-hypercyclic.
We will extend this result to unilateral weighted shifts on Fr\'echet spaces having a Schauder basis by applying Proposition \ref{criterio alap}.

\begin{theorem}\label{equiv backward AP}
	Let $X$ be a separable Fr\'echet space with Schauder basis $\{e_n\}$ and suppose that $B(e_{n+1})=e_n$ is a well defined and continuous backward shift.
	The following are equivalent:
	\begin{itemize}
	    \item [i)]$B$ is $\mathcal{AP}$-hypercyclic;
	    \item [ii)] $B$ is multiply recurrent;
	    \item [iii)]$e_{n_k}\to0$ for some sequence $(n_k)_k\in \AP$ with the following property: given $p,m\in\mathbb N$ there exists $q\in\mathbb N$ such that the arithmetic progression of length $m$, common difference $q$ and initial term $p+q$ is contained in $({n_k})_k$;
	    \item [iv)]$B$ satisfies the $\mathcal{AP}$-hypercyclicity criterion.
	    \item [v)] $T, T^2\ldots ,T^m$ are disjoint hypercyclic for every $m$.
	\end{itemize}  
\end{theorem}

For the proof we will need the following lemma.

\begin{lemma}\label{sucesion alap}
	Let $(n_k)_k\in \AP$ with the property that given $p,m\in\mathbb N$ there exists $q\in\mathbb N$ such that the arithmetic progression of length $m$, common difference $q$ and initial term $p+q$ is contained in $({n_k})_k$ and such that $x_{n_k-j}\to 0$ for every $j\geq 0$. Then there is a sequence in $\mathcal{AP}$, $(m_k)_k$ which contains arbitrarily long arithmetic progressions whose initial term coincides with the common difference, such that $x_{m_k+j}\to 0$ for every $j$.
\end{lemma}
	
\begin{proof}
	Let $\|\cdot\|$ denote the $F$-norm of $X$. By our assumptions, for each $m$, there is an arithmetic progression of length $m$, common difference $q_m$ and initial term $m+q_m$ such that $\|x_{lq_{m}+m-j}\|<\f{1}{m}$ for every $0\leq j\leq m$, $1\leq l\leq m$. 
	Thus, $\|x_{lq_m+j}\|<\f{1}{m}$ for every $0\leq j\leq m$, $1\leq l\leq m$. 
	
	 Let $(m_k)_k$ be the sequence formed by $\cup_m\{lq_m:\, 1\leq l\leq m \}$. Then $(m_k)_k$ is in $\mathcal{AP}$, it has arbitrarily long arithmetic progressions whose initial term coincides with the common difference and satisfies that $\|x_{m_k+j}\|\to 0$ for every $j$.
\end{proof}

\begin{proof}[Proof of Theorem \ref{equiv backward AP}]
Let $\|\cdot\|$ denote the $F$-norm of $X$.

i)$\Rightarrow$ ii) is obvious by definition.

ii)$\Rightarrow$ iii).
Suppose  that $B$ is multiply recurrent. It suffices to show that for each $\epsilon>0$, and each $p,m\in\mathbb N$ there is $q$ such that $\|e_{jq+p}\|<\varepsilon$ for every $1\leq j \leq m$.

Since $\{e_n\}$   is a Schauder basis, there exists $\delta>0$ such that $\|x-e_p\|<\delta$ implies $\|x_ne_n\|<\f{\epsilon}{2}$ for any $n\ne p$ and $|x_p|>\frac12$. Since $B$ is multiply recurrent, there exist $q$ and $x$ such that $\|x-e_p\|<\delta$ and $\|B^{jq}(x)-e_p\|<\delta$ for every $1\leq j \leq m$. Thus, $\|x_n e_n\|<\f{\epsilon}{2}$ for every $n\ne p$ and $|B^{jq}(x)_p|=|x_{jq+p}|>\frac{1}{2}$ for every $0\leq j\leq m$. It follows that $\|e_{jq+p}\|<\epsilon$ for every $1\leq j\leq m$.

iii)$\Rightarrow $ iv).
 Let $X_0=span(e_n:n\in \zN)$ and $S$ be the forward shift defined in $X_0$.  We have for free that $B^n(x)\to 0$ for every $x\in X_0$ and that $BS(x)=x$. 
 It remains to find a sequence $(m_k)_k\in\mathcal{AP}$ with arbitrarily long arithmetic sequences of the form $\{q,2q,\dots,mq\}$ such that $S^{m_k}(x)\to 0$ for every $x\in c_{00}$.


Since $e_{n_k}\to 0$ we have that $e_{n_k-j}=B^{j}(e_{n_k})\to 0$ for every $j$. Thus, by Lemma \ref{sucesion alap}, there is an sequence  $(m_k)_k\in\mathcal{AP}$ with the required property such that $e_{m_k+j}=S^{m_k}(e_j)\to 0$ for every $j$. 

iv)$\Rightarrow $i) follows by Proposition \ref{criterio alap}.	
\end{proof}

Applying a quasiconjugation argument we obtain an analogous result for weighted backward shifts.
\begin{corollary}
	Let $X$ be a Fr\'echet space with Schauder basis $\{e_n\}$ and suppose that $B_\omega(e_{n+1})=\omega_n e_n$ is a well defined and continuous weighted backward shift.
	The following are equivalent:
	\begin{itemize}
	    \item [i)] $B_\omega$ is $\mathcal{AP}$-hypercyclic;
	    \item [ii)] $B_\omega$ is multiply recurrent;
	    \item [iii)] $\prod_{l=1}^{n_k} \omega_l^{-1} e_{n_k}\to 0$ for some sequence   $(n_k)_k\in \mathcal{AP}$ with the following property: given $p,m\in\mathbb N$ there exists $q\in\mathbb N$ such that the arithmetic progression of length $m$, common difference $q$ and initial term $p+q$ is contained in $({n_k})_k$; 
	    \item [iv)] $B_\omega$ satisfies the $\mathcal{AP}$-hypercyclicity criterion.  
	\end{itemize}
	\end{corollary}
Recall that every hypercyclic backward shift is weakly mixing. Thus, every $\AP$-hypercyclic backward shift is weakly mixing but the converse is not true. On the other hand, we will see in Theorem  \ref{alap no weakly} that not every $\AP$-hypercyclic operator is weakly mixing.

Recall also that a backward shift on a Fr\'echet space with basis is mixing if and only if $e_n\to 0$, thus every mixing backward shift is $\AP$-hypercyclic but there are $\AP$-hypercyclic backward shifts that are not mixing. In \cite{kwietniak2017multi} there is an example of a mixing subshift in $\{0,1\}^{\mathbb N}$ that is not $\AP$-transitive and in \cite{puig2017mixing} an example of a mixing linear operator such that $\{T,T^2\}$ is not $d$-transitive (but it is $\AP$-hypercyclic because it is chaotic). Note also that if $T$ is mixing then $T\oplus T^2\oplus \dots\oplus T^m$ is hypercyclic for every $m$. But we don't know the answer to the following. 
\begin{question}\label{pregunta mixing implica AP}
Is every mixing linear operator on a separable Fr\'echet space necessarily $\AP$-hypercyclic? Or equivalently, is any mixing operator multiply recurrent?
\end{question}

\section{Weakly mixing and multiply recurrent operators}

In this section we study the relationship between $\AP$-hypercyclicity and weak mixing for linear operators. It was shown in \cite{kwietniak2017multi} that if $T$ is $\AP$-hypercyclic then $T\oplus T$ is $\AP$-recurrent (in the context of compact dynamical systems, but the proof also works in the non-compact case).  On the other hand, we know that there are weakly mixing operators that are not $\AP$-hypercyclic (see \cite[Proposition 5.8]{CosPar12} or the characterization of $\AP$-hypercyclic backward shifts, Theorem \ref{equiv backward AP}). 
We  show that the converse implication does not hold either. 
We then characterize operators that are both weakly mixing and multiply recurrent.

\subsection*{An $\mathcal{AP}$-hypercyclic operator which is not weakly mixing}
 In 2009 De la Rosa and Read solved one of the most important problems that were open in linear dynamics: they constructed a hypercyclic operator that is not weakly mixing or equivalently that does not satisfy the hypercyclicity criterion \cite{dlRRea09}. The hypercyclicity criterion has many formulations and it is usually implied by some regularity condition. For instance, every chaotic operator or every reiterative hypercyclic operator is weakly mixing. Thus, it would be reasonable to expect that $\AP$-hypercyclicity implies weakly mixing.

On the other hand Bayart and Matheron constructed  examples of non-weakly mixing hypercyclic operators on classical spaces, such as $\ell_p,c_0,\H$ \cite{BayMat07}. They also studied  in \cite{BayMat09Non} non-weakly mixing operators having orbits with a high level of frequency, and proved that if $(m_k)_k$ satisfies that $\lim_k \f{m_k}{k}=+\infty$ 
then there exists a hypercyclic non-weakly mixing operator $T$ on $\ell_1$ satisfying that for each nonempty open set $U$, the recurrence set $N_T(x,U)$ is $O((m_k)_k)$. Note that this result is very tight since if $(\frac{m_k}{k})_k$ were bounded then such a $T$ would be frequently hypercyclic and hence weakly mixing.

We will show next that this result, together with some quantitative upper bounds proved by Gowers for Szemer\'edi's theorem, implies that there are  $\mathcal{AP}$-hypercyclic operators which are not weakly mixing. See also
\cite{cardeccia2020thesis}, where such an operator is explicitly constructed.
\begin{theorem}\label{alap no weakly}
There exists a multiply recurrent and hypercyclic operator on $\ell_1$ that is not weakly mixing.
\end{theorem}
\begin{proof}
Let $f(t):={t}{\sqrt {2}^{-\sqrt{\log\log\log t}}}$ and
$m_l:=[f^{-1}(l)]$ for $l\in\mathbb N$, where $\log$ is the base $2$ logarithm and $[t]$ denotes the integer part of $t\in \mathbb R$.
 Then since $l\sim {m_l}{\sqrt 2^{-\sqrt{\log\log\log m_l}}}$,
 $$
\frac{m_l}{l}\sim \frac{m_l}{\frac{m_l}{\sqrt 2^{\sqrt{\log\log\log m_l}}}}\to \infty,\quad \textrm{as }l\to\infty.
$$
 Thus by \cite{BayMat09Non}, there exist $T$ on $\ell_1$ and $x\in\ell_1$ such that $T$ is not weakly mixing and $N_T(x,U)$ is $O((m_l)_l)$ for every nonempty open set $U$.
 
 Let us show that such an operator $T$ must be $\AP$-hypercyclic.
 Let $r_k(n)$ be the maximum of all $r$ such that there exists $A\subset \{1\dots,n\}$ such that
$|A|=r$  and $A$ does not have an arithmetic progression of length $k.$
 
 It is known by \cite{gowers2001new} that 
 $r_k(n)<\frac{n}{(\log\log n)^{2^{-2^{k+9}}}}-1$. 
 Take $k(n)=[\log\log\sqrt{\log\log\log n}-9]$. 
 Then 
 $$
 r_{k(n)}(n)+1<\frac{n}{(\log\log n)^{2^{-\log\sqrt{\log\log\log n}}}}=\frac{n}{(\log\log n)^{\frac1{\sqrt{\log\log\log n}}}} = \frac{n}{2^{\sqrt{\log\log\log n}}}.
 $$
  Then, since $l>r_{k(m_l)}(m_l)$, there must be an arithmetic progression of length $k(m_l)$ contained in $\{m_1,m_2,\dots,m_l\}$. Moreover if $(n_l)_l=O((m_l)_l)$ then $n_l\le Cm_l$ for some $C>0$ then since $l\sim \frac{m_l}{2^{\sqrt{\log\log\log\sqrt m_l}}}$
$$
\frac{r_{k(n_l)}(n_l)}{l} \le \frac{r_{k(Cm_l)}(Cm_l)}{l} \le \frac{\frac{Cm_l}{2^{\sqrt{\log\log\log Cm_l}}}}{\frac{\tilde Cm_l}{\sqrt 2^{\sqrt{\log\log\log m_l}}}}<1,
$$
for sufficiently large $l$. Therefore $\{n_1,n_2,\dots,n_l\}$ contains an arithmetic progression of length $k(n_l)$ for all sufficiently large $l$.
Consequently, $T$ is $\AP$-hypercyclic.
\end{proof}

\subsection*{Weak mixing and multiple recurrence}
Now we proceed to characterize operators that are both weakly mixing and multiply recurrent. We will show that classical results on weakly mixing operators have an ``$\AP$-analogue''.

\begin{theorem}
The following are equivalent: 
\begin{enumerate}[label=(\roman*)]
    \item $T$ is weakly mixing and multiply recurrent.
       \item $T\oplus T$ is $\AP$-hypercyclic
    \item \emph{(Furstenberg type theorem)}
 $T\oplus T\ldots \oplus T$ is $\AP$-hypercyclic for every $n\in\mathbb N$.
    \item \emph{($T$ is hereditarily $\AP$-hypercyclic)} There is $(n_k)_k\in \AP$ such that for every $\AP$-subsequence $(m_k)_k$ of $(n_k)_k$ there is some $x$ satysfying that $N_T(x,U)\cap (m_k)_k\in \AP$ for  every nonempty open set $U$.
     \item $T$ satisfies the following criterion: there are an $\AP$-sequence $(n_{k,j})_{k\in\mathbb N, 0\le j\le k}=(a_k+jc_k)_{k\in\mathbb N, 0\le j\le k}$, dense sets $X_0,Y_0$ and applications $S_{k,j}:Y_0\to X$ such that for every $x\in X_0$ and every $y\in Y_0$,
\begin{enumerate}
    \item $T^{n_{k,j}}(x)\to 0$,
    \item $(S_{n_{k,0}}+\dots+S_{n_{k,k}})(y)\to 0$,
    \item $T^{n_{k,j}}(S_{n_{k,0}}+\dots+S_{n_{k,k}})y\to y,$ as $k\to\infty$ (independently of the choice of $j\le k$).
\end{enumerate}
     \item For every nonempty open sets $U,V_1,V_2$ and every length $m$ there are $x_1,x_2\in U$ and $a,k\in\mathbb N$ such that $T^{a+jk}(x_i)\in V_i$ for every $j\leq m$.
  \end{enumerate}
\end{theorem}


\begin{proof}
$(i)\Rightarrow(ii)$. Let $U_1,U_2,V_1, V_2$ be nonempty open sets and $m>0$.
Since $T$ is weakly mixing, there are $U_1'\subset U_1$, $V_1'\subset V_1$ and $N$ such that $T^N(U_1')\subset U_2$ and $T^N(V_1')\subset V_2$. 

On the other hand, since $T$ is $\AP$-hypercyclic, there exist $x_1\in U_1'$, and $a,k\in \mathbb N$ such that $T^{a+kj}x_1\in V_1'\subset V_1$ for $j\le m$. Let now $x_2:=T^Nx_1\in T^NU_1'\subset U_2$. Then $T^{a+kj}x_2\in T^NV_1'\subset V_2.$
We have proved that $(x_1,x_2)\in U_1\times U_2$ and for any $j\le m,$
$$
(T\oplus T)^{a+kj}(x_1,x_2)\in V_1\times V_2.
$$

$(ii)\Rightarrow(iii)$
We prove it by induction. Let $(U_j)_j, (V_j)_j$ be open sets, $1\leq j\leq n+1$. Let $m>0$. Hence, there are $N$, $U_{n}'\sub U_{n}$ and $V_{n}'\sub V_{n}'$ such that $T^{N}(U_n')\sub U_{n+1}$ and $T^{N}(V_n')\sub V_{n+1}$. 

By assumption there are $k,a\in\mathbb N$, $x_i\in U_i$, $1\leq i \leq n-1$ and $x_n\in U_n'\sub U_n$ such that $T^{a+jk}(x_i)\in V_i$, $1\leq i\leq n-1$ and $T^{a+jk}(x_n)\in V_{n'}\sub V_n$. Finally we define $x_{n+1}=T^N(x_n)\in U_{n+1}$. Hence, we have that for every $j$, $T^{a+jk}(x_{n+1})=T^{N+a+jk}(x_n)\in T^N(V_n')\sub V_{n+1}.$

$(iii)\Rightarrow(iv)$. Let $(U_j\times V_j)_j$ be a basis of open sets for $X\times X$.
By $(iii)$, for every $N$, there are $a_N,k_N\in\mathbb N$ such that for $1\le l,j\le N,$ there is $x_{N,l}\in U_l$ such that
$$
T^{a_N+jk_N}(x_{N,l})\in V_l,
$$
i.e. $N_{T\oplus\dots\oplus T}((x_{N,1},\dots,x_{N,N}),V_1\times\dots\times V_N)$ contains an arithmetic progression of length $N$.

Let now $(n_k)_k$ be the sequence formed by $\{a_N+jk_N: j\leq N \text { and } N\in\zN\}$. 
  Moreover, by Proposition \ref{Ansari AP}, we may assume that $a_{N+1}>2(a_{N}+Nk_N)$.
Note that every arithmetic progression  of length $m$ contained in $(n_k)_k$ must be contained in $\{a_N+jk_N:j\leq N\}$ for some $N\ge m$.

Let $(m_k)_k\subset(n_k)_k$, $(m_k)_k\in\AP$. 
We will show that there exists an $\AP$-universal vector $x$ for $(T^{m_k})_k$, that is $N_T(x,U)\cap (m_k)_k\in \AP$ for every nonempty open set $U$.

Notice that, in the same way as in Proposition \ref{equivalencias alap}, it is enough to prove that:
\begin{align}
\nonumber\textrm{for each open sets } U,V &\textrm{ and }m>0,  \textrm{ there is }x\in U \textrm{ such that: }\\
\label{equivalencia AP-universal}&N_T(x,V)\cap (m_k)_k \textrm{ has an arithmetic progression of length }m.
\end{align}
Indeed, if \eqref{equivalencia AP-universal} holds and $(V_n)_n$ is a basis of open sets, then we consider
$$\mathcal O_{l}:= \{x: N_{T}(x,V_l)\cap (m_k)_k \text{ admits an arithmetic progression of length } l\}.$$
It turns out that each $\mathcal O_l$ is a dense open set and hence $\cap_l \mathcal O_l$ is dense. Thus, each vector in $\cap_l \mathcal O_l$ satisfies that $N_T(x,U)\cap (m_k)_k\in \AP$ for every open set $U$.

We now show \eqref{equivalencia AP-universal}. Let  $M>0.$

Take $r$ so that $U_r\times V_r\subset U\times V$, and let $\{b+jc:j=1,\dots,m\}\subset(m_k)_k\cap \{a_N+jk_N:j\leq N\}$ such that $N>\max\{r,M\}.$
Thus, for every $1\le j, l\le N$, 
$$
T^{a_N+jk_N}(x_{N,l})\in V_l.
$$
In particular, for every $1\le j\le M$, 
$$
T^{b+jc}(x_{N,r})\in V.
$$

$(iv)\Rightarrow (v)$ Let $x\in X$ such that for every open set $U$, $N_T(x,U)\cap (n_k)_k\in \AP$. Then, for each $k,$ we may find $(b_k+jd_k)_{j\le k}\subset (n_k)_k$ such that 
$$
T^{b_k+jd_k}(x)\in\frac1{k}B_X.
$$
By $(iv)$ there exists $z\in X$ such that for each open set $U$,
$$
N(z,U)\cap \{b_k+jd_k: k\in\mathbb N, j\le k \}\in \AP.
$$
In particular, there exists a sequence $(n_{k,j})_{j\le k}=(a_k+jc_k)_{j\le k}\subset (b_k+jd_k)_{j\le k}$
such that for $1\le j\le k,$
$$
T^{a_k+jc_k}(z)\in B_X(kx,1).
$$
Define $x_{k,j}:=\frac{z}{k}$ for $j\le k$. Then $x_{k,j}\to 0$ and 
$$
T^{n_{k,j}}(x_{k,j})=T^{a_k+jc_k}(z/k)\in  B_X(x,\frac1{k}),
$$
which implies that $T^{n_{k,j}}(x_{k,j})\to x.$

Let now $X_0=Y_0=orb_T(x)$, which are dense in $X$. Thus, if $T^n(x)\in X_0$,
$T^{n_{k,j}}(T^nx)=T^nT^{a_k+jc_k}(x)\to 0$ as $k\to \infty$. 

We define now $S_{n_{k,j}}$ on $Y_0=orb_T(x)$ as
$$
S_{n_{k,j}}(T^nx):=\frac1{k+1}T^nx_{k,j}=\frac1{k+1}T^n\frac{z}{k}.
$$
Then $(S_{n_{k,0}}+\dots+S_{n_{k,k}})(T^nx)=T^n\frac{z}{k}\to 0$ as $k\to\infty$.

Finally, if $j\le k,$
$$
T^{n_{k,j}}(S_{n_{k,0}}+\dots+S_{n_{k,k}})(T^nx)= \frac1{k+1}T^{n_{k,j}}(T^nx_{k,0}+\dots+T^nx_{k,k}) =T^nT^{n_{k,j}}x_{k,j}\to T^nx,
$$
as $k\to\infty.$

$(v)\Rightarrow(i)$ $T$ satisfies the hypercyclicity criterion, thus, by the B\`es-Peris theorem \cite{BesPer99}, $T$ is weakly mixing.

We prove that it is multiply recurrent. Take $U$ a nonempty open set and $m\in \mathbb N.$ 
We know that for each $i_k,j_k\le k$, $y\in Y_0\cap U$,
$$
T^{a_k+j_kc_k}(S_{a_k}+\dots+S_{a_k+kc_k})y\to y.
$$
In particular, for $k$ big enough, $z:=T^{a_k}(S_{a_k}+\dots+S_{a_k+kc_k})y$ belongs to $U$, and 
$$
\{T^{a_k}z,T^{a_k+c_k}z,\dots,T^{a_k+mc_k}z\}\subset U,
$$
that is, $z\in U\cap (T^{-c_k}U)\cap\dots\cap (T^{-mc_k}U)$. 
This implies that $T$ is multiply recurrent.

To finish the proof we show that $(ii)$ and $(vi)$ are equivalent. 

$(ii)\Rightarrow(vi)$ is immediate.

$(ii)\Leftarrow(vi)$.
Let $U_1,U_2,V_1,V_2$ open sets and $m>0$. By hypothesis there is $n_1\in N(U_1,U_2)\cap N(U_1,V_2)$. Applying the hypothesis again we obtain that there are $a$, $k$ and $x_1\in U_1\cap T^{-{n_1}}(U_2), x_2\in U_1\cap T^{-n_1}(U_2)$ such that $T^{a+jk}(x_1)\in V_1$ and $T^{a+jk}(x_2)\in T^{-n} (V_2)$ for every $j\leq m$. We notice that $T^{n_2}(x_2)\in U$ and that $T^{a+jk}(T^n(x_2))\in V_2$ for every $j\leq m$. 
\end{proof}

\section{Infinitely many arithmetic progressions with the same step }
In this short section we study multiple recurrence with the additional property that there are infinitely many arithmetic progressions with the same step contained in the sets of return times. We see that these notions coincide for linear operators but differ for families of operators. We study this concept in connection to a Theorem due to Costakis and Parissis \cite{CosPar12}.

Given a Furstenberg family $\mathcal F$, the following definition was given in \cite{Pui18} (see also \cite{CosPar12} and \cite[Proposition 4.6]{badea2007unimodular}).
\begin{definition}
Given a family $\mathcal F$,
a sequence of operators $(T_n)_n$ is said to satisfy  property $\mathcal P_{\mathcal F}$ if, for each non-empty open set $U$ in $X$, there exists $x\in X$ such that $\{n\in\mathbb N: T_nx\in U\}\in\mathcal F$.
An operator $T$ satisfies $\mathcal P_{\mathcal F}$ if $(T^n)_n$ has the property $\mathcal P_{\mathcal F}$.
\end{definition}

The main result in \cite{CosPar12} proves that if a sequence of scalars $(\lambda_n)_n$ is such that  $\f{\lambda_{n+\tau}}{\lambda_n}\to 1 $ for some $\tau$ and $(\lambda_nT^n)_n$ has the property $\mathcal P_{\overline{\mathcal BD}}$ then $T$  itself is multiply recurrent. The key ingredient is that, via an application of Szemeredi's Theorem, any set $A\in\overline{\mathcal {BD}}$ satisfies that for each $m>0$ there is $k$ such that $A$ contains infinitely many arithmetic progressions of the same step $k$ and length $m$.

\begin{definition}
We will say that $A\in \overline {\AP}$ provided that for every $m$ there is $k$ such that $A$ has infinitely many arithmetic progressions of step $k$ and length $m$.
\end{definition}

Thus, a close look to the proof of \cite[Theorem 3.8]{CosPar12} shows that it can be stated in the following form.
\begin{theorem}[Costakis-Parissis]\label{teo costakis}
Let $(\lambda_n)_n$ be a sequence of complex numbers such that $\frac{\lambda_{n+\tau}}{\lambda_n}\to 1$ for some $\tau$. Then 
$$
(\lambda_nT^n)_n \textrm{ has property } \mathcal P_{\overline {\AP}} \quad\Rightarrow\quad T \textrm{ is  multiply recurrent}.
$$
\end{theorem}

We show next that, for a single operator, all these forms of recurrence are equivalent.
\begin{proposition}\label{equivalencia overline AP}
\begin{itemize}
\item[$(a)$] The following are equivalent.
\begin{itemize}
    \item[(i)] $T$ is multiply recurrent.
    \item[(ii)] $T$  has property $\mathcal P_{{\AP}}.$
    \item[(iii)] $T$ is $\AP$-recurrent.
    \item[(iv)] $T$ has property  $\mathcal P_{\overline {\AP}}.$
    \item[(v)] $T$ is $\overline\AP$-recurrent.
\end{itemize}

\item[$(b)$]  $T$ is $\AP$-hypercyclic operator if and only if $T$ is $\overline {\AP}$- hypercyclic.
\end{itemize}
\end{proposition}
\begin{proof}
We prove $(a)$, the proof of $(b)$ is similar.
Clearly, (iii)$\Rightarrow$(ii)$\Rightarrow$(i)  and (v)$\Rightarrow$(iv)$\Rightarrow$(ii). 

Hence, by Proposition \ref{multiple rec implica AP-rec}, (i), (ii) and (iii) are  equivalent. 

It remains to show that (iii)$\Rightarrow$(v).
Let $U$ be a nonempty open set. By (the proof of) Proposition \ref{multiple rec implica AP-rec}, there is an $\AP$-recurrent vector $x\in U$ such that for every neighbourhood $V$ of $x$ and any $m$, there exists $k$ such that 
$$
x,T^kx,\dots,T^{km}x\,\in V.
$$
Thus, there is an open set $V'\subset V$ such that $T^{jk}V'\subset V$ for $0\le j\le m$. Moreover, since $x$ is a recurrent vector, there is a sequence $(n_l)_l$ such that $T^{n_l}x\in V'$. This implies that $x$ is an $\overline\AP$-recurrent vector because for $j\le m$ and any $l,$
$$
T^{n_l+jk}x\,\in T^{jk}V'\subset V.
$$
 \end{proof}
 This shows that we may see Costakis-Parissis' result (Theorem \ref{teo costakis}) as a result about property $\mathcal P_{\overline {\AP}}$.
 \begin{corollary}
 Let $(\lambda_n)_n$ be a sequence of complex numbers such that $\frac{\lambda_{n+\tau}}{\lambda_n}\to 1$ for some $\tau$. Then 
$$
(\lambda_nT^n)_n \textrm{ has property } \mathcal P_{\overline {\AP}} \quad\Rightarrow\quad T \textrm{ has property } \mathcal P_{\overline {\AP}}.
$$
 \end{corollary}

It is now natural to ask  if  property $\mathcal P_{\overline {\AP}}$ 
can be replaced by property $\mathcal P_{{\AP}}$ in the above corollary, i.e. is it true that if 
$\frac{\lambda_{n+\tau}}{\lambda_n}\to 1$ for some $\tau$,
$$
(\lambda_nT^n)_n \textrm{ has property } \mathcal P_{{\AP}} \quad\Rightarrow\quad T\textrm{ has property }\mathcal P_{ {\AP}}\,?
$$

We will now prove that this is not true.
First we need the following proposition.
\begin{proposition}\label{prop lambda B PAP}
Let $X=\ell_p$, $B$ the backward shift and $\lambda_n\to \infty$ such that $\frac{\lambda_{j'k}}{\lambda_{jk}}\to 0$ as $k\to\infty,$  for every $j>j'$. Then $(\lambda_n B^n)_n$ has the $\mathcal P_{\AP}$ property and it is $\AP$-universal.
\end{proposition}
Applying the above proposition with  $\lambda_n=e^{\sqrt n}$ we have that  $(\lambda_n B^n)_n$ has property $\mathcal P_{\AP}$. On the other hand, since the backward shift is not multiply recurrent, we conclude. 
\begin{corollary}\label{No ejemplo AP}
There exist an operator $T$ and a sequence $(\lambda_n)_n$ such that  $\f{\lambda_{n+1}}{\lambda_n}\to 1$, $(\lambda_n T^n)_n$ has property $\mathcal P_\AP$ but $T$ is not multiply recurrent.
\end{corollary}
Note that, in contrast to Proposition \ref{equivalencia overline AP}, 
Corollary \ref{No ejemplo AP} shows that for  families of operators properties $\mathcal P_{\AP}$ and $\mathcal P_{\overline{\AP}}$ are not equivalent.
\begin{proof}[Proof of Proposition \ref{prop lambda B PAP}]

Let $U=B(y,\varepsilon)$ an open ball of radius $\varepsilon>0$ where $y\in c_{00}$. Let $m$ be any natural number and $k=k(m)$ to be determined.
Let $T_n=\lambda_n B^n$.

We consider $\tilde y= \sum_{j=0}^{m} \frac{S^{jk}(y)}{\lambda_{jk}},$ where $S$ is the forward shift and we have adopted the convention $\lambda_0:=1.$

Let $0\leq l\leq m$.
If $k>supp(y)$ we have that
\begin{align*}
\lambda_{lk}T_{lk}(\tilde y)&=y +\sum_{j=l+1}^{m} \f{\lambda_{lk}S^{(j-l)k}(y)}{\lambda_{jk}}
\end{align*}

Therefore, if $k$ is big enough so that
 $|\f{\lambda_{j'k}}{\lambda_jk}|<\frac{\varepsilon}{m \|y\|}$ for every $ j'<j\le m$,
then
we have that for every $l\leq m$,
$$\|\lambda_{lk}T^{lk}(\tilde y)-y\|=\|\sum_{j=l+1}^{m} \f{\lambda_{lk}S^{(j-l)k}(y)}{\lambda_{jk}}\|<\varepsilon.$$
The proof of $\AP$-universality follows similarly.
\end{proof}

\begin{remark}
In \cite{Pui18}, a sequence $(\lambda_nT^n)_n$ satisfying property $\mathcal P_{\overline{\mathcal {BD}}}$ was characterized  in terms of a special kind of recurrence for $T$ called topologically $\mathcal D$-recurrence with respect to $(\lambda_n)_n$.

Let us consider the following type of recurrence with respect to $(\lambda_n)_n$ (which is similar, but simpler, to topologically $\mathcal D$-recurrence):
for each $U$, there exists  $x$ satisfying that for each  $m$, there is some $k$ such that 
\begin{align}\label{equivalencia P_AP estilo puig}
    {card}\left\{a:\lambda_aT^ax\in\bigcap_{i=0}^mT^{-ik}(U)\right\}=\infty.
\end{align}
Then, it is easy to show that, for a sequence $(\lambda_n)_n$ with 
$\frac{\lambda_{n+\tau}}{\lambda_n}\to 1$ for some $\tau$, 
\begin{center}
$(\lambda_nT^n)_n$ has property $\mathcal P_{\overline {\AP}}$ if and only if  \eqref{equivalencia P_AP estilo puig} holds.
\end{center}
\end{remark}
Indeed, suppose that $(\lambda_nT^n)_n$ has property $\mathcal P_{\overline {\AP}}$. Let $U$ be an open set. Let $M\in\mathbb N$, take $\delta>0$ and $V$ a nonempty open set such that $V+B_\delta\subset U$, and let $x$ be a vector satisfying  $\{n\in\mathbb N: \lambda_nT^nx\in U\}\in\overline {\AP}$.  Then in particular, there is $k$ such that for infinitely many $a$'s,
\begin{align*}
    \lambda_{a+i\tau k}T^{a+i k}x\in V,\quad\textrm{for }0\le i\le \tau M & \Rightarrow \frac{\lambda_{a+i\tau k}}{\lambda_a}\lambda_{a}T^{a+i\tau k}x\in V,\quad\textrm{for }0\le i\le M\\
    &\Rightarrow \lambda_{a}T^{a+i\tau k}x\in U,\quad\textrm{for }0\le i\le M \textrm{ and  }a\ge a(U,\delta,M,\tau)\\
    &\Rightarrow \lambda_{a}T^{a}x\in \bigcap_{i=0}^M T^{-i\tau k}U,\quad\textrm{for }0\le i\le M \textrm{ and }a\ge a(U,\delta,M,\tau).\\
\end{align*}
 The proof of the converse is similar.
 

  For any sequence $(\lambda_n)_n,$ the multiple recurrence of $T$ is clearly implied by the  recurrence defined in \eqref{equivalencia P_AP estilo puig}. Thus, the implication proved in the above remark, together with Szemeredi's Theorem, provides a simpler proof of \cite[Theorem 3.8]{CosPar12}, although the main idea is the same.

\end{document}